\documentclass{amsart}
\usepackage{bm,color}

\title
[Abstract formulatio ...]
{Abstract formulation of the Miura transform}

\author{Yoritaka Iwata}

\address[Y. Iwata]{ Faculty of Chemistry, Materials and Bioengineering, Kansai University}
\email{iwata$\_$phys@08.alumni.u-tokyo.ac.jp}

\thanks{This article has been presented in ICRAAM 2020 conference, Kuala Lumper, Feb. 2020.
The author is grateful to Prof. Emeritus Hiroki Tanabe for valuable comments.
The referee's comment on the Ref.~\cite{18iwata-1} is acknowledged for giving an idea for initiating this research.
This work was partially supported by JSPS KAKENHI Grant No. 17K05440.}

\theoremstyle{plain}
\newtheorem{theorem}{Theorem}[section]
\newtheorem{corollary}[theorem]{Corollary}
\newtheorem{lemma}[theorem]{Lemma}

\begin{document}
\begin{abstract}
Miura transform is known as the transformation between  Korteweg de-Vries equation and modified Korteweg de-Vries equation. Its formal similarity to the Cole-Hopf transform has been noticed. This fact sheds light on the logarithmic type transformations as an origin of certain kind of nonlinearity in the soliton equations. In this article, based on the logarithmic representation of operators in infinite-dimensional Banach spaces, a structure common to both Miura and Cole-Hopf transforms is discussed. In conclusion, the Miura transform is generalized as the transform in abstract Banach spaces, and it is applied to the higher order abstract evolution equations.
\end{abstract}


\maketitle



\setcounter{section}{0} 

\section{Introduction}
The Korteweg-de-Vries equation (KdV equation, for short)  and the modified Korteweg de-Vries equation (mKdV equation, for short) are known as a nonlinear equations holding the soliton solutions.
Let $u$ and $v$ be the solutions of KdV equation and mKdV equation, respectively.
Both equations, which are known as soliton equations, are nonlinear, and the functions $u$ and $v$ are assumed to be the general solutions that satisfy
\[\begin{array}{ll} \label{kdvs}
{\rm [KdV]} \quad \partial_t u - 6 u \partial_x u + \partial_x^3 u = 0, \vspace{2.5mm} \\
{\rm [mKdV]} \quad\partial_t v - 6 v^2 \partial_x v + \partial_x^3 v = 0
\end{array} \]
without identifying the details such as the initial and boundary conditions of the mixed problem.
For a recent result associated with the well-posedness of the KdV equations, the existence and uniqueness of the solution of semilinear KdV equations in non-parabolic domain is obtained in \cite{20benia} by using the parabolic regularization method, the Faedo-Galerkin method, and the approximation of a non-parabolic domain by a sequence of regularizable subdomains.
Meanwhile interesting studies on the family of KdV-type equations have been recently carried out in \cite{13ruggieri}.
Let a set of all the real numbers be denoted by ${\mathbb R}$.
Although $u$ and $v$ are functions of $t \in {\mathbb R}$ and $x \in {\mathbb R}$, they are not apparently shown if there is no confusion.
The Miura transform~\cite{68miura} ${\mathcal M}:u \to v$ reads
\begin{equation} \begin{array}{ll} \label{miura}
u = \partial_x v + v^2, 
\end{array} \end{equation}
which is formally the same as the Riccati's differential equation of a variable $x$ if $u$ is assumed to be a known function.
In this article, the Miura transform is generalized as the transform in the abstract spaces.
The essence of several nonlinear transforms are pined downed within the theory of abstract equations defined in a general Banach spaces.
In conclusion, the structure of the general solutions of second order abstract evolution equations are presented in association with the Miura transform.

\section{Operator logarithm as nonlinear transform}
\subsection{Nonlinear transform associated with the Riccati's equation}
Following the method for solving the Riccati's equation, the logarithmic type transform appears as
\begin{equation} \begin{array}{ll} \label{colehopf}
v = \psi^{-1} \partial_x \psi, 
\end{array} \end{equation}
which corresponds to $v = \partial_x \log \psi $ if $ \log \psi$ and its derivative are well-defined.
This is formally the same as the Cole-Hopf transform \cite{51cole, 50hopf, 18iwata-1}.
By applying this transform to the Miura transform, the Miura transform is written by
\begin{equation} \begin{array}{ll} \label{miura}
u = \partial_x v + v^2 \vspace{2.5mm} \\
 ~ = \partial_x ( \psi^{-1} \partial_x \psi) + ( \psi^{-1} \partial_x \psi)^2 \vspace{2.5mm} \\
 ~ = -  \psi^{-2} (\partial_x  \psi )^2 + \psi^{-1}  ( \partial_x^2 \psi) 
  +  \psi^{-1}  (\partial_x \psi)   \psi^{-1} ( \partial_x \psi). 
\end{array} \end{equation}
If $\psi$ and $ \partial_x \psi$ commute,
\begin{equation} \label{miura-mod} \begin{array}{ll} 
\qquad u  = -  (\partial_x  \psi )^2  \psi^{-2} +  ( \partial_x^2 \psi) \psi^{-1}   +    (\partial_x \psi)^2  \psi^{-2}
 =   ( \partial_x^2 \psi) \psi^{-1}    \vspace{2.5mm} \\
\Leftrightarrow \quad \partial_x^2 \psi = u \psi .   
\end{array}  \end{equation}
This is the second order evolution equation in which $u$ plays a role of  infinitesimal generator, and the evolution direction is fixed to $x$.
It is remarkable that, after the combination with the Cole-Hopf transform, the Miura transform ${\mathcal M}:u \to \psi$ is a transform between nonlinear KdV equation and linear equation. 
In other words, it provides the transform between the evolution operator and its infinitesimal generator.
In the following the obtained transform from $u$ to $\psi$ is called the combined Miura transform.

\subsection{Miura transform and Cole-Hopf transform}
It is worth differentiating Eq.~\eqref{colehopf} for clarifying the identity of the Miura transform.
Under the commutation assumption, the formal calculation without taking the differentiability into account leads to 
\begin{equation} \label{2ndlog} \begin{array}{ll} 
 \partial_x^2  (\log \psi )  := \partial_x ( \psi^{-1} \partial_x \psi )  =   ( \partial_x^2 \psi) \psi^{-1}     -  ( \psi^{-1} \partial_x  \psi )^2,  
\end{array}  \end{equation}
where the first term of the right hand side corresponds to the combined Miura transform, and the second term of the right hand side is the square of the Cole-Hopf transform.
It simply means that $\partial_x^2  (\log \psi ) $ being defined by the right hand side of \eqref{2ndlog} can be regarded by defining the the combined Miura transform and Cole-Hopf transform simultaneously.
As is well known in the theory of integrable systems, $\partial_x^2  (\log \psi )$ corresponds to one typical type of Hirota's methods \cite{71hirota}, thus a typical type of linear to nonlinear transformation.
This type is known to be associated with the B\"{a}cklund transform and KP theory (for a textbook, see \cite{91jackson}).

\subsection{Logarithmic representation of infinitesimal generators}
Let $X$ be a Banach space, $B(X)$ is a space of bounded linear operators on $X$, and $Y$ be a dense Banach subspace of $X$.
The Cauchy problem for the first order abstract evolution equation of hyperbolic type \cite{70kato, 73kato} is defined by
\begin{equation} \begin{array}{ll} \label{eq01}
d u(t)/dt  - A(t) u(t) = f(t),  \qquad t \in [0,T] \vspace{2.5mm} \\
u(0) = u_0
\end{array} \end{equation}
in $X$, where $A(t): Y \to X$ is assumed to be the infinitesimal generator of the evolution operator $U(t,s) \in B(X)$ satisfying the strong continuity and the semigroup property:
\[ U(t,s) = U(t,r)U(r,s) \]
for $0 \le s \le r \le t < T$.
$U(t,s)$ is a two-parameter $C_0$-semigroup of operator (for definition, see \cite{65yosida, 66kato, 79tanabe}) that is a generalization of one-parameter $C_0$-semigroup and therefore an abstract generalization of the exponential function of operator.
If $A(t)$ is confirmed to be an infinitesimal generator, then the solution $u(t)$ is represented by $u(t) = U(t,s) u_s$ with $u_s \in X$ for a certain $0 \le s \le T$ (cf. Hille-Yosida Theorem; for example, see \cite{65yosida, 66kato, 79tanabe}). 
Then, for a certain complex number $\kappa$, the alternative bounded infinitesimal generator  $a(t,s) = {\rm Log} (U(t,s) + \kappa I)$ to $A(t)$ is well defined \cite{17iwata-1,17iwata-3}, where ${\rm Log}$ denotes the principal branch of logarithm.

\begin{lemma} [Logarithmic representation of infinitesimal generators \cite{17iwata-1}]
\label{lem1}
Let $t$ and $s$ satisfy $0 \le t,s \le T$, and $Y$ be a dense subspace of $X$.
Let $a(t,s) \in B(X)$ be defined by $a(t,s) = {\rm Log} U(t,s)$.
If $A(t)$ and $U(t,s)$ commute, infinitesimal generators $\{ A(t) \}_{0 \le t \le T}$ are represented by means of the logarithm function; there exists a certain complex number $\kappa \ne 0$ such that
\begin{equation} \label{logex} \begin{array}{ll}
 A(t) ~ u =  (I - \kappa e^{-a(t,s)})^{-1}~  \partial_{t}  a(t,s)  ~ u, 
\end{array} \end{equation}
where $u$ is an element of a dense subspace $Y$ of $X$, and the logarithm of function is defined by the Riesz-Dunford integral.
\end{lemma}

\begin{proof}
Only formal discussion is given here (for the detail, see \cite{17iwata-1, 19iwata}).   
Since $a(t)$ is defined by $a(t,s) = {\rm Log} (U(t,s) + \kappa I) $, $\partial_{t}  a(t,s) =  (U(t,s) + \kappa I )^{-1} \partial_t U(t,s)  $.
\[
(U(t,s) + \kappa I )  \partial_{t}  a(t,s)   = (U(t,s) + \kappa I ) (U(t,s) + \kappa I )^{-1} \partial_t U(t,s) .
\]
Under the commutation relation between $U(t,s)$ and $A(t)$,
\[ \begin{array}{ll}
A(t) ~ u := \partial_t {\rm Log} U(t,s) u \vspace{1.5mm}  \\
 = U(t,s)^{-1} \partial_t U(t,s) u \vspace{1.5mm}  \\
 = U(t,s)^{-1} (U(t,s) + \kappa I )  \partial_t a(t,s) ~ u  \vspace{1.5mm}  \\
  = (I + \kappa ( e^{a(t,s)}- \kappa I  )^{-1}  )  \partial_t a(t,s) ~ u  \vspace{1.5mm}  \\
 =  (  e^{a(t,s)}- \kappa I     + \kappa I  )   ( e^{a(t,s)}- \kappa I  )^{-1} \partial_{t}  a(t,s)  ~ u   \vspace{1.5mm}  \\
 =  ( I - \kappa e^{-a(t,s)} )^{-1} \partial_{t}  a(t,s)  ~ u ,
\end{array} \] 
where $u$ is an element in $Y$.

\end{proof}

The commutation assumption is trivially satisfied if $A(t)$ is independent of $t$.
Equation (\ref{logex}) is the logarithmic representation of infinitesimal generator $A(t)$.
This representation is the generalization of the Cole-Hopf transform \cite{51cole,50hopf}.

\section{Main result}
\subsection{Generalization of Miura transform}
Let ${\mathcal X}$ be a Banach space, and ${\mathcal Y}$ be a dense Banach subspace of ${\mathcal X}$.
By focusing on establishing the definition of infinitesimal generator, the discussion is limited to the autonomous case. 
The Cauchy problem for the second order abstract evolution equation of hyperbolic type (for example, see \cite{10brezis}) is defined by
\begin{equation} \begin{array}{ll} \label{eq02}
d^2 u(t)/dt^2  - {\mathcal A} (t) u(t) = 0,  \qquad t \in [0,T] \vspace{2.5mm} \\
u(0) = u_0
\end{array} \end{equation}
in ${\mathcal X}$ and $ {\mathcal A}(t):{\mathcal Y} \to {\mathcal X}$.
Let the Cauchy problem be solvable; i.e., it admits the well-defined evolution operator $ {\mathcal U}(t,s) \in B({\mathcal X})$ satisfying the strong continuity and the semigroup property:
\[  {\mathcal U}(t,s) =  {\mathcal U}(t,r)  {\mathcal U}(r,s) \]
for $0 \le s \le r \le t < T$.
The solution is represented by $u(t) = {\mathcal U}(t,s) u_s$ with $u_s \in {\mathcal X}$ for a certain $0 \le s \le T$.
For the second order equation, ${\mathcal U}(t,s)$ is not equal to the abstraction of $\exp (\int {\mathcal A}(t) dt)$, so that ${\mathcal A}(t)$ is not the infinitesimal generator of  ${\mathcal U}(t,s)$, and ${\mathcal A} (t)$ is the infinitesimal generator of $\exp (\int {\mathcal A}(t) dt)$ instead.
In the following, the combined Miura transform is shown to be equivalent to the logarithmic representation for the infinitesimal generator of the second order abstract evolution equations.

The master equation of \eqref{eq02} is also written as a system of equations:
\begin{equation} \left\{ \begin{array}{ll} \label{eq03}
d u(t)/dt  -  v(t) = 0,  \vspace{2.5mm} \\
d v(t)/dt  - {\mathcal A} (t) u(t) = 0,  
\end{array} \right. \end{equation}
where, by focusing on  the representation of $v(t)$, $v(t)$ is formally represented by $v(t) = \partial_t  {\mathcal U}(t,s) v_s$ for a certain $v_s \in D(\partial_t  {\mathcal U}(t,s))$ that is compatible with the original $u(t) = {\mathcal U}(t,s) u_s$ for a certain $u_s \in D({\mathcal U}(t,s)) = {\mathcal X}$.

\begin{lemma}[Logarithmic representation of the derivative] \label{lem2}
Let $\kappa$ be a certain complex number.
For the evolution operator of Eq,~\eqref{eq02}, let $ {\mathcal U} (t,s)$ be included in the $C^1$ class in terms of variables t and s,
and the first order derivative $ {\mathcal V}(t,s) := \partial_t  {\mathcal U}(t,s)$ be further assumed to be bounded on ${\mathcal X}$ and strongly continuous for $0 \le t, s \le T$.
Then, for $ {\mathcal V}(t,s)$,
\begin{equation} \begin{array}{ll} \label{2ndrep2}
\partial_t {\rm Log}  {\mathcal V}(t,s)
: =      ( I - \kappa e^{-{\hat \alpha}(t,s)} )^{-1} ~ \partial_{t}  {\hat \alpha}(t,s)  
\end{array} \end{equation}
is well defined if ${\mathcal V}(t,s)$ and  $\partial_t   {\mathcal V}(t,s) $ commute, where ${\hat \alpha}(t,s): {\mathcal X} \to  {\mathcal X}$ is an operator defined by ${\hat \alpha}(t,s) = {\rm Log} ( {\mathcal V}(t,s) + \kappa I)$.
\end{lemma}

\begin{proof}
The statement follows from Lemma~\ref{lem1}.
\end{proof}

On the other hand, another logarithmic representation
\begin{equation} \begin{array}{ll} \label{2ndrep3}
\partial_t {\rm Log}  {\mathcal U}(t,s)
: =      ( I - \kappa e^{-\alpha(t,s)} )^{-1} ~ \partial_{t}  \alpha(t,s)  
\end{array} \end{equation}
is trivially well-defined by the assumption.
According to the representations (\ref{2ndrep2}) and (\ref{2ndrep3}), the abstract version of the Miura transform is obtained as the product of two logarithmic representations.

\begin{theorem}[Abstract formulation of the Miura transform]
Let $t$ and $s$ satisfy $0 \le t,s \le T$, $\kappa$ be a certain complex number, and $ {\mathcal Y}$ and  $ {\mathcal Y}'$ be a dense subspace of $ {\mathcal X}$.
The operator $\partial_t \alpha(t,s)$ is assumed to be a closed operator from ${\mathcal Y}$ to ${\mathcal X}$, and $\partial_t {\hat \alpha}(t,s)$ is assumed to be a closed operator from ${\mathcal Y}'$ to ${\mathcal X}$.
If  ${\mathcal U}(t,s)$, $\partial_t {\mathcal U}(t,s)$, and  $\partial_t^2 {\mathcal U}(t,s)$ commute with each other within a properly given domain space, the operators $\{ {\mathcal A}(t) \}_{0 \le t \le T}$ are represented by means of the logarithm function; there exists a certain complex number $\kappa \ne 0$ such that
\begin{equation} \label{logex2-1} \begin{array}{ll}
{\mathcal A}(t) ~ u 
 =   ( I - \kappa e^{-{\hat \alpha}(t,s)} )^{-1} \partial_{t}  {\hat \alpha}(t,s) ~
    ( I - \kappa e^{-\alpha(t,s)} )^{-1} \partial_{t}  \alpha(t,s)   ~     u,   
\end{array} \end{equation}
for an element $u$ of a dense subspace $ \{ u \in {\mathcal Y}; ~   ( I - \kappa e^{-\alpha(t,s)} )^{-1}  \partial_{t}  \alpha(t,s)  u \subset D( \partial_{t}  {\hat \alpha}(t,s)) \}$ of ${\mathcal X}$, and
\begin{equation} \label{logex2-2} \begin{array}{ll}
{\mathcal A}(t) ~ {\hat u} 
 =    ( I - \kappa e^{-\alpha(t,s)} )^{-1} \partial_{t}  \alpha(t,s)  ( I - \kappa e^{-{\hat \alpha}(t,s)} )^{-1} \partial_{t}  {\hat \alpha}(t,s) ~
    ~     {\hat u}, 
\end{array} \end{equation}
for an element ${\hat u}$ of a dense subspace $\{ {\hat u} \in {\mathcal Y}'; ~  ( I - \kappa e^{-{\hat \alpha}(t,s)} )^{-1} \partial_{t}  {\hat \alpha}(t,s) {\hat u} \subset D(\partial_{t}  \alpha(t,s)) \}$ of ${\mathcal X}$. 
\end{theorem}

\begin{proof}
The autonomous equation \eqref{eq02}, which corresponds to the abstract form of the Miura transform \eqref{miura-mod}, is written by
\[ \begin{array}{ll} 
 \partial_{t}^2  {\mathcal U}(t,s) u = {\mathcal A}(t)   {\mathcal U}(t,s) u
\end{array} \]
for any $u \in {\mathcal X}$, so that it follows that
$\partial_{t}^2  {\mathcal U}(t,s)  = {\mathcal A}   {\mathcal U}(t,s) $
is valid as an operator equation.
Under the assumption of commutative property between $ {\mathcal A} $ and ${\mathcal U}(t,s)$, the operator $ {\mathcal A} $ is represented by
\[ \begin{array}{ll} 
{\mathcal A}(t) =    {\mathcal U}(t,s)^{-1}  \partial_{t}^2  {\mathcal U}(t,s)  \vspace{2.5mm} \\
\quad =    {\mathcal U}(t,s)^{-1}  \partial_{t}  {\mathcal U}(t,s)  
~ ( \partial_{t}  {\mathcal U}(t,s) )^{-1}   \partial_{t}^2  {\mathcal U}(t,s), 
\end{array} \]
where the former part $ {\mathcal U}(t,s)^{-1}  \partial_{t}  {\mathcal U}(t,s)  $ and the latter part $ ( \partial_{t}  {\mathcal U}(t,s) )^{-1}   \partial_{t}^2  {\mathcal U}(t,s)$ of the right hand side correspond to the logarithmic representation $  ( I - \kappa e^{-\alpha(t,s)} )^{-1} \partial_{t}  \alpha(t,s) $ and $  ( I - \kappa e^{-{\hat \alpha}(t,s)} )^{-1} \partial_{t}  {\hat \alpha}(t,s) $ respectively.
In this equation ${\mathcal U}(t,s)$, $\partial_t {\mathcal U}(t,s)$, and  $\partial_t^2 {\mathcal U}(t,s)$ are assumed to commute with each other, so that the logarithmic representation of $A(t)$ follows.
\end{proof}

In particular, for the commutation between two generally-unbounded operators, the intermediate domain space can be different depending on the order of  operators.
Here is a reason why the two.representations \eqref{logex2-1} and \eqref{logex2-2} are obtained depending on the order of operators.

\subsection{Second order abstract evolution equations}

\begin{corollary}[Logarithmic representation of infinitesimal generator]
Let operators $\partial_t {\mathcal U}(t,s)$ and  $\partial_t {\mathcal V}(t,s) = \partial_t^2 {\mathcal U}(t,s)$ satisfy the sectorial property.
If either Eq.~\eqref{logex2-1} or Eq.~\eqref{logex2-2} is well-defined, then their square roots being represented by either
\begin{equation} \label{logex3-1}
\pm {\mathcal A}(t)^{1/2} = 
\pm \left\{    ( I - \kappa e^{-{\hat \alpha}(t,s)} )^{-1} ~ \partial_{t}  {\hat \alpha}(t,s)
  ( I - \kappa e^{-\alpha(t,s)} )^{-1} \partial_{t}  \alpha(t,s) \right\}^{1/2}   
\end{equation}
or 
\begin{equation} \label{logex3-2}
\pm {\mathcal A}(t)^{1/2} = 
\pm \left\{ ( I - \kappa e^{-\alpha(t,s)} )^{-1} \partial_{t}  \alpha(t,s)   ~
                      ( I - \kappa e^{-{\hat \alpha}(t,s)} )^{-1} ~ \partial_{t}  {\hat \alpha}(t,s) \right\}^{1/2}  
\end{equation}
are the infinitesimal generators of Eq.~\eqref{eq02} in the sense that the solution of $d^2 u(t)/dt^2  = {\mathcal A} (t) u(t)$ is represented by 
\begin{equation} \label{linearcom} u(t) = {\mathcal U}(t,s) u_0
= \exp \left[ + {\mathcal A}(t)^{1/2} \right] u_{+}  +   \exp  \left[ - {\mathcal A}(t)^{1/2} \right]  u_{-} , 
\end{equation}
where $u_+$ and $u_-$ are the elements of ${\mathcal X}$.
The representation \eqref{logex3-1} is valid if \eqref{logex2-1} is true, and the representation \eqref{logex3-2} is valid if \eqref{logex2-2} is true.
\end{corollary}
\begin{proof}
Under the commutation relation, the autonomous equation \eqref{eq02} is also formally factorized as
\[ \begin{array}{ll} 
 \left( \partial_{t}  {\mathcal U}(t,s) ^{1/2}  +  {\mathcal A}^{1/2}  {\mathcal U}(t,s) ^{1/2} \right)
 \left( \partial_{t} {\mathcal U}(t,s) ^{1/2} -  {\mathcal A} ^{1/2}  {\mathcal U}(t,s) ^{1/2}  \right) u = 0
\end{array} \]
for any $u \in {\mathcal X}$.
It leads to the decomposition such that
\[ \left\{ \begin{array}{ll} 
 \partial_{t}  {\mathcal U}(t,s) ^{1/2} u_+  +  {\mathcal A}^{1/2}  {\mathcal U}(t,s) ^{1/2} u_+ = 0 ,  \vspace{2.5mm} \\
 \partial_{t}  {\mathcal U}(t,s) ^{1/2} u_-  -  {\mathcal A}^{1/2}  {\mathcal U}(t,s) ^{1/2} u_- = 0.
\end{array} \right. \]
The representation shown in Eq.~\eqref{linearcom} is understood.

It is necessary to confirm the possibility of defining the fractional power of ${\mathcal A}$.
The possibility of defining square root of operator is justified if it is possible to define the exponential of
\[ \begin{array}{ll} 
{\rm Log} \left[ {\mathcal A}(t)^{1/2} \right] := 
\frac{1}{2} \left\{   {\rm Log} \left[  \partial_t  {\rm Log}  {\mathcal U}(t,s)  \right] 
+ {\rm Log} \left[  \partial_t  {\rm Log}  {\mathcal V}(t,s) \right]
\right\}, 
\end{array}  \]
where note that the logarithms of the right hand side are the formal form, and it does not matter whether they are well defined or not.
Since $\alpha(t,s) \in B({\mathcal X})$ and ${\hat \alpha}(t,s) \in B({\mathcal X})$ are the logarithms of operator, here the logarithm of logarithm is necessary to be well-defined.
According to the commutation assumption between $\partial_t {\mathcal U}(t,s)$ and  $ \partial_t {\mathcal V}(t,s) =  \partial_t^2 {\mathcal U}(t,s)$, the exponential of each  logarithm of the right hand side are independently well-defined.
According to the commutation relation between  ${\mathcal U}(t,s)$, $\partial_t {\mathcal U}(t,s)$ and  $ \partial_t {\mathcal V}(t,s) =  \partial_t^2 {\mathcal U}(t,s)$, the exponential function of right hand side is equal to
\[ \begin{array}{ll} 
 \left[  \partial_t  {\rm Log}  {\mathcal U}(t,s)  \right] ^{1/2} 
 ~  \left[  \partial_t  {\rm Log}  {\mathcal V}(t,s)  \right] ^{1/2}
 =    \left[  \partial_t  {\rm Log}  {\mathcal V}(t,s)  \right] ^{1/2}
  ~  \left[  \partial_t  {\rm Log}  {\mathcal U}(t,s)  \right] ^{1/2},
\end{array}  \]
and each square root is well-defined by the sectorial assumption (for the sectorial property, see Sec.2.10 of Chap.5 in \cite{66kato}), where the logarithms of ${\mathcal U}(t,s)$ and $ {\mathcal V}(t,s)$ leading to the definition of fractional powers of ${\mathcal U}(t,s)$ and $ {\mathcal V}(t,s)$ are valid as seen in Eqs.~\eqref{2ndrep2} and \eqref{2ndrep3}.
Consequently, the logarithmic representations of infinitesimal generators $\pm {\mathcal A}(t)^{1/2} $ are true.
\end{proof}

\section{Summary}
In this article the Miura transform is generalized in the following sense:
\begin{itemize}
\item it is not only the transform between the KdV and mKdV equations; \vspace{1.5mm}
\item the spatial dimension of the equation is not necessarily equal to 1; \vspace{1.5mm}
\item the differential in Eq.~\eqref{logex2-1} is not necessarily for the spatial variable $x$;
\end{itemize}
where, in terms of applying to theory of higher order abstract evolution equations, the variable is taken as $t$ in this article. 
For the preceding work dealing with the general choice of the evolution direction, see \cite{19iwata-proc, 19iwata}. 
Consequently the generalized Miura transform is obtained as the product of two logarithmic representations of operators in a general Banach space framework and they are applied to clarify the structure of the general solutions of second order abstract evolution equations defined in finite and infinite dimensional Banach spaces.
Since the linear operator $\mathcal{A}$ is a generalized concept of matrices, the presented result potentially includes any matrix situations.

\end{document}